\documentclass[12pt]{amsart}

\usepackage{amsmath}
\usepackage{amsfonts}
\usepackage{amssymb}
\usepackage{latexsym}
\usepackage{graphicx}
\usepackage{enumerate}
\usepackage{multirow}
\usepackage{subfigure}

\usepackage{graphicx,color,dsfont}
\usepackage{enumitem}
\usepackage{fourier}

\newtheorem{theorem}{Theorem}

\newtheorem{lemma}[theorem]{Lemma}

\newtheorem{definition}{Definition}
\newtheorem{assumption}{Assumption}

\usepackage{mathtools}

\makeatletter
\DeclareRobustCommand\widecheck[1]{{\mathpalette\@widecheck{#1}}}
\def\@widecheck#1#2{%
    \setbox\z@\hbox{\m@th$#1#2$}%
    \setbox\tw@\hbox{\m@th$#1%
       \widehat{%
          \vrule\@width\z@\@height\ht\z@
          \vrule\@height\z@\@width\wd\z@}$}%
    \dp\tw@-\ht\z@
    \@tempdima\ht\z@ \advance\@tempdima2\ht\tw@ \divide\@tempdima\thr@@
    \setbox\tw@\hbox{%
       \raise\@tempdima\hbox{\scalebox{1}[-1]{\lower\@tempdima\box
\tw@}}}%
    {\ooalign{\box\tw@ \cr \box\z@}}}
\makeatother

\newcommand{\ep}{\epsilon}

\newcommand{\ga}{\gamma}

\DeclareMathOperator{\sech}{sech}

\DeclareMathOperator{\spn}{span}

\newcommand{\ds}{\displaystyle}

\newcommand{\be}{\begin{equation}}
\newcommand{\ee}{\end{equation}}
\newcommand{\bes}{\begin{equation*}}
\newcommand{\ees}{\end{equation*}}
\newcommand{\mand}{\quad \text{and}\quad}

\newcommand{\R}{{\bf{R}}}

\newcommand{\T}{{\bf{T}}}
\newcommand{\Z}{{\bf{Z}}}

\newcommand{\p}{\partial}
\newcommand{\cl}{{\mathcal L}}
\newcommand{\ck}{{\mathcal K}}
\newcommand{\rone}{\mathbf R}

\newcommand{\beq}{\begin{equation}}
\newcommand{\eeq}{\end{equation}}

\newcommand{\de}{\delta}

\newcommand{\la}{\lambda}

\newcommand{\si}{\sigma}

\newcommand{\cj}{{\mathcal J}}
\newcommand{\dpr}[2]{\langle #1,#2 \rangle}

\renewcommand{\L}{{\mathcal{L}}}

\newcommand{\W}{{\mathcal{W}}}

\renewcommand{\tilde}{\widetilde}
\renewcommand{\hat}{\widehat}

\newcommand{\bunderbrace}[2]{%
  \begin{array}[t]{@{}c@{}}
  \underbrace{#1}\\
  #2
  \end{array}
}

\title{Small amplitude traveling waves in the full-dispersion Whitham equation}
\author[Atanas Stefanov]{Atanas Stefanov} 

\address{ Department of Mathematics,
University of Kansas,
1460 Jayhawk Boulevard,  Lawrence KS 66045--7523, USA}

\email{stefanov@ku.edu}

\thanks{ Stefanov  is partially  supported by  NSF-DMS under grant  \# 1614734.
Wright  is partially  supported by  NSF-DMS under grant  \# 1511488.}

\author[J. Douglas Wright]{J. Douglas Wright}
\address{ Department of Mathematics,
Drexel University, 3141 Chestnut St,
Philadelphia,  PA 19104, USA}
\email{jdoug@math.drexel.edu}

\begin{document}

\begin{abstract}
In this article, we provide an alternative way to construct small amplitude traveling waves for general Whitham type equations, in both periodic and whole line contexts. More specifically,   Fourier analysis techniques allow us to reformulate the problem to the study of waves that are small and regular perturbations of well-understood ODE's.  In addition, rigorous  stability of these waves is established. 
 \end{abstract}
\maketitle

\section{Introduction} 
The equation, 
\begin{equation}
\label{l:10} 
u_t+ \W u_x+2 u u_x=0, \ \ \widehat{\W u}(k)=\sqrt{\frac{\tanh(k)}{k}} \hat{u}(k)
\end{equation}
was proposed by Whitham \cite{W} as an alternative model to the ubiquitous Korteweg-de Vries approximation ($u_t + u_{xxx} + 2 uu_x = 0$) for water waves.
In particular, \eqref{l:10} is driven by the non-local operator $\W$, which  (modulo some rescalings) 
gives the ``full-dispersion" relation for the corresponding water waves equation.
   It also allows, in sharp contrast with the KdV model,  for  wave breaking (\cite{W1}), a desirable realistic  feature for such models.

In this article we study a  generalization of \eqref{l:10}. More specifically, we allow for 
 the following sort of  ``pseudo-differential equations of Whitham type":
\be\label{W}
u_t + (Lu + n(u))_x = 0, \quad u = u(x,t) \in \R,\quad x \in \R \mand \ t \in \R,
\ee
where $n:\R \to \R$ is purely nonlinear. The operator $L$ is a  Fourier multiplier operator with symbol $m$. That is 
$$
\widehat{L f}(k)=m(k) \hat{f}(k)
$$
where $\hat{f}(k)$ is the Fourier transform of $f(x)$.
Precise conditions on $m$ and $n$ will be set forth below, but the prototypical choices will be of course   
$$
m(k) = \sqrt{\tanh(k)/k} \mand  n(u) = u^2,
$$
which then leads us to the original model \eqref{l:10}.  The dynamical properties of \eqref{l:10}, such as local well-posedness, 
wave breaking among others (for \eqref{l:10} as well as for  some more general versions, similar to \eqref{W}) have been thoroughly explored in recent years. We do not review these developments here, as the main focus of the current work lies in the existence and the properties of a class of special solutions, namely traveling waves.  

More specifically, 
we make the traveling wave ansatz  $u(x,t) = w(x-\nu t)$, where $\nu \in \R$ is   as yet undetermined wave speed. After one integration
	we arrive at:
\be\label{TWE1}
(\nu  - L)w = n(w).
\ee
The question for existence and the corresponding properties of traveling waves, that is solutions of \eqref{TWE1}, in either the whole line or periodic context, has been the subject of numerous papers  over the last ten years. We mention the papers   \cite{EK1, EK}, where the question for existence periodic waves is investigated, both rigorously and numerically.  Traveling waves in a model with weak surface tension were considered in \cite{GW10}. Finally, in the {\it tour de force}, \cite{EGW}, the authors have constructed (through an involved constrained variational with penalization  construction), traveling waves for the whole line problem, with speeds slightly bigger than the sonic speed $\nu=1$. The question for stability of these waves, mostly in the periodic context, was considered recently in \cite{SKCK}. It should be noted that both in the analytical and numerical results discussed herein and elsewhere, it appears that there is some natural barrier for the  the wave speeds, $1<\nu<1.141...$, which is still not  fully understood.  Thus, the ``slightly supersonic'' assumption in   these papers appears to be well-warranted. 
 The methods  in  these papers are varied and rather technical. In some cases, the analysis is supplemented by numerical simulations, which is justified given the lack of precise formulas, even in the classical case \eqref{l:10}.  
 
 In this article,  we take  a slightly different point of view. A rescaling of the problem,  together with some Fourier analysis reformulates the problem in such a way that the governing equations for the traveling waves are small and regular perturbations of well-understood
ordinary differential equations. Then we use an implicit function theorem to prove the existence of solutions when the scaling parameter is small. The main ideas of the method are inspired by the work of Friesecke \&  Pego \cite{FP} and Friesecke \& Mikikits-Leitner \cite{FM} on traveling waves in Fermi-Pasta-Ulam-Tsingou lattices, whose governing
equations are nonlocal in a way similar to those we study here.

 \subsection{Assumptions and main results} We make the following assumption regarding $n(u)$. 
\begin{assumption}\label{n ass} There exists $\delta_*>0$ such that 
the nonlinearity $n: (-\delta_*,\delta_*) \to \R$ is $C^{2,1}$ (that is, its second derivative exists and is uniformly Lipschitz continuous) and satisfies
$$
n(0) = n'(0) = 0 \mand n''(0) > 0.
$$
\end{assumption}

And here is our assumption on the multiplier $m$, which is a sort of combination of convexity near zero with boundededness 
for large $k$:
\begin{assumption}\label{m ass}
The multiplier $m:\R \to \R$ is  even and there exists $k_*>0$ 
which has the following properties:
\begin{itemize}
\item $m$ is $C^{3,1}$ (that is, its third derivative exists and is uniformly Lipschitz continuous) on $[-k_*,k_*]$, $m(0)>0$ and  
\be 
\label{m2 bound} 
m_2:=\max_{|k|\le k_*} m''(k) < 0.
\ee
In particular $m''(0)<0$. 
\item  
\be\label{upperbound}
m_1:=\sup_{k \ge k_*} m(k)  < m(0).
\ee
\end{itemize}
\end{assumption}
An important quantity that will arise in the analysis is 
\be\label{this is gamma}
\gamma:= - {n''(0)  \over m''(0)}>0, 
\ee
by Assumptions \ref{n ass} and \ref{m ass}.
Both Assumption \ref{n ass} and \ref{m ass} are easily verified for the choices which give the full-dispersion Whitham equation \eqref{l:10}.
Here are our main results. Note that  our construction provides explicit leading  term both for the waves speeds and the traveling wave profile\footnote{In principle, one could compute explicitly  the next terms, up to any  degree of accuracy} . 
\begin{theorem} 
\label{theo:10} 
The following hold when Assumptions \ref{n ass} and \ref{m ass} are met. There exists $\ep_0>0$, so that for every $\ep\in (0, \ep_0)$, there is a traveling wave solution 
$u(x,t)= \ep^2 W_\ep(\ep(x-\nu_\ep t))$ of \eqref{W}. Moreover,  $W_\ep\in H^1_{even}(\rone)$, 
\begin{eqnarray}
\label{a:105}
\nu_\ep &=&  m(0) - {1 \over 2} m''(0) \ep^2,  \\
\label{107}
W_\ep(x) &=& \frac{3}{2\ga}  \sech^2\left(\frac{x}{2}\right)+O_{H^1(\R)}(\ep^2).
\end{eqnarray}

 In addition, assume the boundedness of $m$. Then, the waves $ \ep^2 W_\ep(\ep(x-\nu_\ep t))$ are in fact spectrally stable, for all small enough values of $\ep$. 
\end{theorem}
{\bf Remarks:} 
\begin{enumerate}
\item Assuming higher regularity of $n$, say $C^{l+2,1}(\rone)$, we  have that $W_\ep\in H^l(\rone)$. 

\item In the proof, we can actually verify the non-degeneracy of the solution $\ep^2 W_\ep(\ep x)$ in the sense that the linearized operator has kernel spanned exactly by the group of symmetries\footnote{in this case, the only symmetry is the translation in the $x$ variable}. By general results for Hamiltonian systems, see for example Theorem 5.2.11 in \cite{KP}, the spectral stability  should imply orbital stability as well. Unfortunately, the conditions in Theorem 5.2.11 in \cite{KP} are not exactly met, since the 
anti self-adjoint portion of the linearization,  $\cj=\p_x$ is not boundedly invertible. This   is likely only a technical issue and we expect  orbital stability to hold as well. 
\end{enumerate} 

We also prove the existence of periodic ``cnoidal" solutions of \eqref{W}.
\begin{theorem} 
\label{theo:10P} 
The following hold when Assumptions \ref{n ass} and \ref{m ass} are met. There exists
$P_0 > 0$ such that the following holds for all $P >P_0$.
There exists $\ep_P>0$, so that for every $\ep\in (0, \ep_P)$ there is a $2P/\ep$-periodic, even, non-zero traveling wave solution 
$u(x,t)= \ep^2 W_{P,\ep}(\ep(x-\nu_\ep t))$ of \eqref{W}. Moreover,  $W_{P,\ep}\in H^1_{even}(\T_{P})$, 
\begin{eqnarray}
\label{a:1055}
\nu_\ep &=&  m(0) - {1 \over 2} m''(0) \ep^2,  \\
\label{1077}
W_{P,\ep}(x) &=&  \phi_P(x)+O_{H^1(\T_P)}(\ep^2),
\end{eqnarray} 
where $\phi_P$ is the unique even, non-zero $2P$-periodic solution of
$-\phi_P''+\phi_P - \gamma \phi_P^2 = 0$. 

In addition, assume the boundedness of $m$.
For $0<\epsilon\ll1$, the waves $W_{P,\epsilon}$ are 
 spectrally and orbitally stable, with respect to co-periodic 
 perturbations (that is   perturbations of the same period $2P/\ep$). 
\end{theorem}
\subsection{Conventions}

By $H^s(\R)$ we mean the usual $L^2$-based order $s$ Sobolev space defined on $\R$. 
By $H^s(\T_P)$ we mean the 
usual $L^2$-based order $s$ Sobolev space of periodic functions with period $2P$.
Restricting attention only to even functions in the above results in the spaces $H^s_{even}(\R)$
and $H^s_{even}(\T_P)$. 
If $X$ is a Banach space then $B(X)$ is the space of bounded linear maps from $X$ to itself, endowed with the usual
norm.

For a function $f \in H^s(\R)$ we use the following normalizations for the Fourier transform and its inverse:
$$
\hat{f}(k)=\ds  \frac{1}{2\pi}\int_\R f(x) e^{- i x k} dx \mand f(x) =\ds   \int_\R \hat{f}(k) e^{ i x k} dk.
$$
For a function $f \in H^s(\T_P)$ we use the following normalizations for the Fourier series and its inverse:
$$
\hat{f}(k):={1 \over 2P}\int_{-P}^{P} f(x) e^{-ik  \pi x/P} dx  \mand f(x) = \sum_{k \in \Z} \hat{f}(k) e^{ik \pi x/P}.
$$

If $X$ is a Banach space and $q_\ep$ is an $\epsilon$ dependent quantity in $X$, we write
$$
q_\ep = O_X(\ep^p) 
$$
if there exists $\ep_0$ and $C>0$ such that
$$
\|q_\ep\|_{X} \le C \ep^p
$$
for $0 < \ep \le \ep_0$.

\section{Existence of small solutions} 
We present a detailed proof for  the whole line case.  The result for the periodic waves, which proceeds in an almost identical fashion,  is proved in Section \ref{sec:2.3}. 

Our approach consists of introducing and analyzing a rescaled system, which   is then showed to approximate the standard 
equation which gives the traveling wave solutions for KdV.
\subsection{The rescaled system}

We make the ``long wave/small amplitude/nearly supersonic" scalings
$$
w(y) = \ep^{2} W(\ep y) \mand \nu = m(0) - {1 \over 2} m''(0) \ep^2
$$
where $0 < \ep \ll1$.
With this, \eqref{TWE1} becomes
\be\label{TWE2}
\left( m(0) - {1 \over 2} m''(0) \ep^2 - L_\ep \right) W = \ep^{-2} n( \ep^2 W)
\ee
where $L_\ep$ is the Fourier multiplier operator with symbol \be
\label{mep}
 m_\ep(k) = m(\ep k).
 \ee 
Since $n(u)$ is $C^{2,1}$ by assumption, Taylor's theorem tells us that
$$
\ep^{-2} n( \ep^2 W) = {\ep^2 \over 2} n''(0) W^2 + \ep^4 \rho_\ep(W)
$$
with 
\be\label{rho est}
|\rho_\ep(W)| \le C |W|^3 \mand |\p_x[\rho(W(x))]|\leq C |W'(x)| |W^2(x)|
\ee
when $|W| \le \delta_*/\ep^2$. Thus \eqref{TWE2} becomes:
\be\label{TWE3}
\left( m(0) - {1 \over 2} m''(0) \ep^2 - L_\ep \right) W = {\ep^2 \over 2} n''(0) W^2  + \ep^4 \rho_\ep(W).
\ee

Assumption  \ref{m ass} implies the following result. 
\begin{lemma} \label{hinge}
Given Assumption \ref{m ass}, there exists $C>0$ such that
\be\label{mult est}
\sup_{K \in \R} \left \vert {\ep^2 \over m(0)-{1 \over 2} m''(0) \ep^2 - m(\ep K)} + {1 \over {1 \over 2}m''(0)(1+ K^2)}\right \vert  \le C\ep^2
\ee
when $\ep$ is sufficiently close to zero.
\end{lemma}
We postpone the technical proof for the Appendix \ref{assorted proofs}, below.
Note however that quite a bit of information is packed into this Lemma. The first piece is that it guarantees that \\ 
$
\left( m(0) - {1 \over 2} m''(0) \ep^2 - L_\ep \right)
$
has a bounded inverse. And so we can rewrite \eqref{TWE3} as:
\be\label{TWE33}
\bunderbrace{W - \ep^2 \left( m(0) - {1 \over 2} m''(0) \ep^2 - L_\ep \right)^{-1} \left( {1\over 2} n''(0) W^2 + \ep^2 \rho_\ep(W)\right)}{\Phi(W,\ep)} = 0.
\ee
Our goal is to resolve \eqref{TWE33}, at least for $0<\ep\ll1$. To do so, we will rely on the implicit function theorem and as such we need the behavior of the limiting system at $\ep=0$. 

Lemma \ref{hinge} implies that
\be\label{limit}
\ep^2 \left( m(0) - {1 \over 2} m''(0) \ep^2 - L_\ep \right)^{-1} = -\frac{2}{m''(0)}(1- \partial_x^2)^{-1} + O_{B(X)} (\ep^2)
\ee
where $X$ is either $H^s(\R)$ or $H^s(\T_P)$.
Thus, if we set $\ep = 0$ in \eqref{TWE33}, we   get:
\be\label{TWE4}
W -{\gamma (1- \partial_X^2)^{-1}} W^2 = 0
\ee
or rather 
\be\label{TWE5}
-W''+  W -\gamma W^2=0. 
\ee
Here $\gamma>0$ is given above in \eqref{this is gamma}.

\subsection{Existence of localized traveling waves}
Observe that \eqref{TWE5}, and so \eqref{TWE4},  has (a unique!) non-zero even localized solution, namely 
\beq
\label{a:110}
W(X) =\sigma (x) :=  \frac{3}{2\ga}  \sech^2\left(\frac{x}{2}\right). 
\end{equation}
In other words,  we have $
\Phi(\sigma,0) = 0.
$

Linearization of \eqref{TWE5} about $\sigma(x)$ results in the 
 self-adjoint operator $$\cl:=-\p_x^2+1 - 2\ga\sigma,$$ which is well-studied in the literature. It is known to have exactly one negative eigenvalue, a single eigenvalue at zero, spanned by $\sigma'$,  and outside of these two directions, the operator $\cl$  is strictly positive. 

\subsubsection{Solvability of \eqref{TWE33}}

If we compute $\ck:=D_W\Phi(\sigma,0)$ we get
$$
\ck= Id  - 2 \gamma (1-\partial_x^2)^{-1} \left(\sigma \cdot \right).
$$
The following lemma is proved in \cite{FP}:
\begin{lemma}
\label{le:19}
$\ck:  L^2_{even}(\rone)\to L^2_{even}(\rone)$ is bounded and has a bounded inverse.  Also, 
$\ck:  H^1_{even}(\rone)\to H^1_{even}(\rone)$ is bounded and invertible. 
\end{lemma}

Here is a brief explanation of the proof.
It is by now a classical result that $(1-\partial_x^2)^{-1} \left(\sigma \cdot \right) :L^2(\rone)\to L^2(\rone)$ 
and indeed $(1-\partial_x^2)^{-1} \left(\sigma \cdot \right) : H^1(\rone)\to H^1(\rone)$   is a compact operator. Thus, the set $\sigma(\ck)\setminus\{1\}$ 
has only eigenvalues of finite multiplicity. Note that when restricted to the even (and also odd subspaces), $\ck$ acts invariantly, that is $\ck: H^1_{even}(\rone)\to H^1_{even}(\rone)$.  We claim that $\ck$ is invertible on $H^1_{even}(\rone)$. Indeed, assuming otherwise, it must be, by the Fredholm alternative, that there is an eigenfunction $f_0\in H^1_{even}: \ck f_0=0$. One quickly realizes that this implies $f_0\in H^2(\rone)$ and $\cl f_0=0$. This is a contradiction, since $f_0\in Ker[\cl]=span[\sigma']$, which then implies that $f_0$ is an odd function. 

We use the following version of the implicit function theorem:
\begin{theorem}
\label{theo:impl} 
Let $X$ be a Banach space and suppose that $\Phi : X \times \R \to X$ has the following properties: (a) $\Phi$ is continuously differentiable (b)
$\Phi(x_*,\mu_*) = 0$ and (c) $D_x \Phi(x_*,\mu_*)$  has bounded inverse from $X$ to $X$ then there exists a neighborhoods $U$ of $x_*$ and $M$ of 
$\mu_*$ and 
differentiable function $\chi: M \to U$ such that $\Phi(\chi(\mu),\mu) = 0$ and $\Phi(x,\mu) = 0$ iff $x = \chi(\mu)$ for all $(x,\mu) \in U \times M$.
\end{theorem}

%
%
%

According to Theorem \ref{theo:impl}, the solvability of \eqref{TWE33}, that is $\Phi(W, \ep)=0$, holds. Indeed, by our previous considerations, $\Phi(\sigma, 0)=0$, the functional $\Phi: H^1_{even}(\R)\times \rone\to H^1_{even}(\R), s>\frac{1}{2}$ is continuously differentiable. In addition, $\ck=D_W\Phi(\sigma,0): H^1_{even}(\rone)\to H^1_{even}(\rone)$ is invertible, according to Lemma \ref{le:19}. This gives a family of solutions, say $W_\ep\in H^{1}_{even}(\rone)$, at least in a small neighborhood of $\ep \in (0, \ep_0)$, $\ep_0<1$. 
That $W_\ep - \sigma$ is $O_{H^1(\R)}(\ep^2)$ follows in routine way from \eqref{TWE33}, \eqref{limit} and \eqref{rho est}.
This finishes the proof of the existence part of Theorem \ref{theo:10}.  

{\bf Remark:} Note that with the current assumptions on $m$, one cannot obtain a higher regularity results on $W_\ep$, since the operator $ \left( m(0) - {1 \over 2} m''(0) \ep^2 - L_\ep \right)^{-1}$ cannot be guaranteed to be smoothing\footnote{and in fact, for the Whitham example, where $m(k)=\sqrt{\frac{\tanh(k)}{k}}$ it is not smoothing}.  
We can however claim higher regularity, by essentially the same arguments as above, once we know a higher regularity of the remainder term 
$\rho_\ep(z)=\frac{n(\ep^2 z)-\frac{n''(0)}{2} \ep^4 z^2}{\ep^3}$ or, what is the same, a higher regularity of the nonlinearity $n$. Indeed, assuming $n\in C^{l+2,1}(\rone)$, we obtain $\rho \in C^{l,1}(\rone)$ and then, we can claim  that the map $\Phi: H^l(\rone)\times \rone\to H^l(\rone)$ is continuously differentiable. Since $\ck$ will also be invertible on $H^l(\rone)$, an application of the implicit function theorem will produce a solution $W_\ep\in H^l(\rone)$.

\subsection{Existence of periodic traveling waves}
\label{sec:2.3} 
Return attention to \eqref{TWE5}. In addition to the solitary wave solution $\sigma(X)$, this equation has a one-parameter family of even periodic solutions.  
While there are explicit formulas available for these solutions (\cite{FM}) in terms of the elliptic functions ``${\textrm{cn}}$" (hence the nomenclature ``cnoidal" waves) we 
do not need these formulas here.  Instead, we summarize the properties of such waves.
\begin{theorem}\label{cnoidal}
For all $\gamma > 0$ there exists $P_0>0$ and a family functions $\left\{\phi_P(x)\right\}_{P>P_0}$  with the following properties
\begin{enumerate}
\item $\phi_P(x)$ is $C^\infty$, non-constant and even.
\item $\phi_P(x)$ is periodic with principal period $2P$.
\item $W(x) = \phi_P(x)$ solves \eqref{TWE5} (and thus \eqref{TWE4})
\item The  kernel of $$
\L_P := - \partial_x^2 + 1- 2 \gamma \phi_P
$$
(as an operator in $H^s(\T_P)$)
is exactly $\spn\left\{\phi_P'(x) \right\}$.
\end{enumerate}
\end{theorem}

This theorem tells us that $\Phi(\phi_P,0) = 0$. Our strategy for continuing such solutions to $\ep >0$ via the implicit
function theorem is not terribly  different than the one used for the localized waves above.
If we compute $\ck_P:=D_W\Phi(\phi_P,0)$ we get
$$
\ck_P= Id  - 2 \gamma (1-\partial_x^2)^{-1} \left(\phi_P \cdot \right).
$$
In \cite{FM} (their Lemma 5.1) the following is shown:
\begin{lemma}
\label{le:19P}
$\ck_P:  L^2_{even}(\T_P)\to L^2_{even}(\T_P)$ is bounded and has a bounded inverse.  Also, 
$\ck_P:  H^1_{even}(\T_P)\to H^1_{even}(\T_P)$ is bounded and invertible. 
\end{lemma}
This follows from part (3) of Theorem \ref{cnoidal} and the argument is
 very much the same as the proof of Lemma \ref{le:19}. At this stage we have appeal to the implicit function theorem as above
and arrive at the conclusions of Theorem \ref{theo:10P}.

\section{Proof of Theorem \ref{theo:10} : the stability of the small Whitham waves}
Now that we have constructed the solutions $W_\ep$ for $0<\ep\ll 1$, let us address the question for their   stability. We first linearize around the traveling wave solution. 
\subsection{The linearized problem and stability}
We take the perturbation of the solution $\ep^2 W_\ep (\ep(x-\nu t))$ in the form $u=\ep^2(W_\ep (\ep(x-\nu t))+v(\ep t, \ep(x-\nu t)))$. Plugging in this ansatz in the equation \eqref{W} and ignoring terms of order $O(v^2)$ and transforming $x-\nu t\to x$,  we obtain the following linearized system
\beq
\label{a:10}
v_t+\p_x[L_\ep v - \nu v+n'(\ep^2 W_\ep) v]=0.
\eeq
Introduce  the linearized operator 
$$
\cl_\ep:=-L_\ep + \nu - n'(\ep^2 W_\ep).
$$
Passing to the time independent problem via the map  $v(t,x)\to e^{\la t} z(x)$, we arrive at the eigenvalue problem 
\beq
\label{a:20}
\p_x \cl_\ep z=\la z
\eeq
It is then time to introduce the notion of stability. 
\begin{definition}
\label{defi:10} 
We say that the traveling wave  $\ep^2 W_\ep (\ep(x-\nu t))$ is spectrally stable, if the eigenvalue problem \eqref{a:20} does not have non-trivial solutions $(\la, z): \Re\la>0,   z \in L^2(\rone)$. 

We say that the solution is orbitally (non-linearly) stable, if for every $\sigma>0$, there exists $\de=\de(\si, \ep)>0$, so that whenever $u_0\in H^1(\rone): 
\|u_0 -  \ep^2 W_\ep(\ep \cdot)\|_{H^1}<\de$, then the solution $u$, with initial data $u_0$, 
$$
\inf_{y\in \rone}\|u(t, \cdot)- \ep^2 W_\ep (\ep(\cdot+y-\nu t))\|_{H^1(\rone)}<\si.
$$
\end{definition}
Next, we discuss the instability index count theory, which gives sufficient (and in many cases necessary) conditions for stability/instability, both spectral and orbital. We mostly follow the general theory, as developed in \cite{LZ}, although earlier relevant results are available, see \cite{KKS, KKS2, KP, KS}. 
\subsection{Instability index theory}
\label{sec:3.2}
For the eigenvalue problem
\beq
\label{a:30}
\cj \cl f=\la f
\end{equation}
make the following assumptions regarding $\cl, \cj$: 
\begin{enumerate}
\item $\cl^*=\cl$, so that $L\in B(X,X^*)$ for some real Hilbert space\footnote{In the most common applications, $X=H^s, s>0$ is a Sobolev space of positive order, while $X^*=H^{-s}$ and one has  $X=D(\cl)\subset L^2 \subset X^*$}  $X$, i.e. $\dpr{\cl u}{v}: X\times X\to {\mathbf C}$ is continuous. 
\item $dim(Ker[\cl])<\infty$ and there is the $\cl$ invariant decomposition of the space $X$, 
$$
X=X_- \oplus Ker[\cl]\oplus X_+,
$$
where $dim(X_-)<\infty$, and for some $\de>0$, 
$\cl|_{X_-}\leq -\de$, $\cl|_{X_+}\geq \de>0$. 
\item   $\cj: D(\cj)\subset X^* \to X$, $\cj^*=-\cj$. 
\end{enumerate} 
Moreover,  introduce the Morse index   $n^-(\cl)=dim(X_-)$, an integer.   
 Consider the generalized eigenspace at zero for the operator $\cj \cl$, that is   
$E_0=\{u\in X: (\cj \cl)^k u=0, k\geq 1 - \textup{integer} \}$. Clearly, $Ker[\cl]$ is a (finite dimensional) subspace of $E_0$ and one can complete it: 
$E_0=Ker[\cl]\oplus \tilde{E}_0$. Then, 
$$
k_0^{\leq 0}:=\max\{dim(Z): Z\  \textup{subspace of}\  \tilde{E}_0: \dpr{\cl z}{z}\leq 0, z\in Z\}.
$$
Under these assumptions, it was proved (see Theorem 2.3, \cite{LZ}) that\footnote{A much more precise result is contained in  Theorem 2.3, \cite{LZ}, but we state this corollary, as it is enough for our purposes} 
\begin{equation}
\label{a:40}
k_{unstable}\leq n^-(\cl)- k_0^{\leq 0}(\cl). 
\end{equation} 
where $k_{unstable}$ is the number of (non-trivial) unstable solutions to \eqref{a:30}, that is pairs  $(\la, z)$ with $\Re\la>0, z \in X$. 
In the next section, we apply this theory to the linearized problem \eqref{a:20}. 
 \subsection{Stability analysis for the small Whitham waves} 
 \label{sec:3.3} 
 For the eigenvalue problem \eqref{a:20}, we have $\cj=\p_x$, which is anti self-adjoint,  while clearly $\cl_\ep: \cl^*_\ep=\cl_\ep$ is a bounded symmetric operator, if we assume the boundedness of its symbol $m$. 
 
 We will establish below that $\cl_\ep$ has, at least for small enough values of $\ep$,  a single and simple negative eigenvalue (i.e. $n^-(\cl_\ep)=1$), while its kernel is one dimensional and it is in fact spanned by $W'_\ep$. Assuming that for the moment, let us proceed to establish a sufficient condition for the stability. 
 According to \eqref{a:40}, $k_{unstable}\leq 1- k_0^{\leq 0}$. Thus,  the stability of the solitary waves 
 $\ep^2 W_\ep(\ep x)$, will be established, once we show that\footnote{and hence $k_0^{\leq 0}(\cl_\ep)=1$, since the left hand side of \eqref{a:40} is non-negative.}
 $k_0^{\leq 0}(\cl_\ep)\geq 1$. 
 
 To this end, we can  identify an element in $gKer(\p_x \cl_\ep)\setminus Ker[\p_x \cl_\ep]$. Note that  $Ker[\p_x \cl_\ep]=Ker[\cl_\ep]=span\{W'_\ep\}$. In addition, 
 $W_\ep \perp  W'_\ep $, whence $W_\ep \perp Ker[\cl_\ep]$. Thus,  $\Psi_\ep:=\cl_\ep^{-1}[W_\ep]$ is well-defined.  Since,  
 $$
 (\p_x \cl_\ep)^2[\Psi_\ep]= \p_x \cl_\ep \p_x[W_\ep ]= \p_x \cl_\ep[W'_\ep]=0, 
 $$
we have that $\Psi_\ep \in gKer(\p_x \cl_\ep)\setminus Ker[\p_x \cl_\ep]$. According to the definition of $k_0^{\leq 0}(\cl_\ep)$, we will have established $k_0^{\leq 0}(\cl_\ep)\geq 1$, once we verify that 
$$
0>\dpr{\cl_\ep \Psi_\ep}{\Psi_\ep}=\dpr{\cl_\ep^{-1}[W_\ep]}{\ep^2 W_\ep}.
$$
Thus, we will need to verify the negativity of the Vakhitov-Kolokolov type quantity 
\beq
\label{a:50}
\dpr{\cl_\ep^{-1}[W_\ep]}{W_\ep}<0,
\eeq
once we check that for all small enough $\ep$, $n^-(\cl_\ep)=1$, $Ker[\cl_\ep]=span\{W'_\ep\}$. We do this in the next Lemma. 
\begin{lemma}
\label{le:a10}
There exists $\ep_0>0$ so that for all $\ep\in (0, \ep_0)$,   $n^-(\cl_\ep)=1$, $Ker[\cl_\ep]=span\{W'_\ep\}$. 
\end{lemma}
\begin{proof}
Start by taking a sufficiently large $\mu>0$, to be specified later.   We will construct the operator  
$\left(\ep^{-2} \cl_\ep+\mu\right)^{-1}$ for all small enough $\ep$. Indeed, since 
$$
n'(\ep^2 W_\ep)= n''(0)\ep^2 W_\ep+O_{H^1}(\ep^4)=n''(0)\ep^2 \sigma+O_{H^1}(\ep^4),
$$
where $\sigma$ is the explicit $sech^2$ function, see \eqref{a:110}. We have  
\begin{eqnarray*}
 \ep^{-2} \cl_\ep+\mu &=& \ep^{-2}[\cl_\ep+\mu \ep^2]=\ep^{-2}[-L_\ep+\nu-\ep^2 n''(0) \sigma+\mu \ep^2+O_{H^1}(\ep^4)]=\\
 &=&  [Id -[ n''(0) \sigma-\mu +O_{H^1}(\ep^2)] \ep^2 (\nu-L_\ep)^{-1}]\ep^{-2}(\nu-L_\ep). 
\end{eqnarray*}
Recall now that the operator $\ep^{2} (\nu-L_\ep)^{-1}$ is associated with the multiplier $\frac{\ep^2}{  m(0)-{1 \over 2} m''(0) \ep^2-m(\ep k)}$. So, according to Lemma \ref{hinge} (and more precisely \eqref{mult est}), 
\beq
\label{a:70}
\ep^{2} (\nu-L_\ep)^{-1}=-\frac{2}{m''(0)} (1-\p_x^2)^{-1}+O_{B(L^2)}(\ep^2).
\eeq
Thus, 
\begin{eqnarray*}
 \ep^{-2} \cl_\ep+\mu  &=&   \left(Id+\frac{2}{m''(0)} [n''(0) \sigma-\mu+O_{H^1}(\ep^2)](1-\p_x^2)^{-1}\right)\ep^{-2}(\nu-L_\ep) =\\
 &=&  \left(Id +2[-\ga \sigma -\frac{\mu}{m''(0)}+O_{H^1}(\ep^2)](1-\p_x^2)^{-1}\right)\ep^{-2}(\nu-L_\ep). 
\end{eqnarray*} 
Note however 
\begin{eqnarray*}
\cl-\frac{2 \mu}{m''(0)}+O_{H^1}(\ep^2) &=& 1-\p_x^2-2\ga \sigma -\frac{2 \mu}{m''(0)}+O_{H^1}(\ep^2)=\\
&=& \left[Id +2[-\ga \sigma -\frac{\mu}{m''(0)}+O_{H^1}(\ep^2)](1-\p_x^2)^{-1}\right](1-\p_x^2). 
\end{eqnarray*} 
Now, we select $\mu>0$ large and $\ep\ll1$, so that $ \cl-\frac{2 \mu}{m''(0)}+O_{H^1}(\ep^2) $ is invertible. This is possible, since $-\frac{2 \mu}{m''(0)}>0$ and $\cl$ is bounded from below\footnote{and in fact it has a single negative eigenvalue}. Moreover, $( \cl-\frac{2 \mu}{m''(0)}+O_{H^1}(\ep^2))^{-1}: L^2\to H^{2}$. Thus, we can write 
\begin{eqnarray*}
\left[Id +2[-\ga \sigma -\frac{\mu}{m''(0)}+O_{H^1}(\ep^2)](1-\p_x^2)^{-1}\right]^{-1}= (1-\p_x^2)( \cl-\frac{2 \mu}{m''(0)}+O_{H^1}(\ep^2))^{-1}:L^2\to L^2.
\end{eqnarray*} 
Hence, we can invert (by means of the previous formula and \eqref{a:70})
\begin{eqnarray*}
(\ep^{-2} \cl_\ep+\mu)^{-1} &=& \ep^{2} (\nu-L_\ep)^{-1} \left[Id +2[-\ga \sigma -\frac{\mu}{m''(0)}+O_{H^1}(\ep^2)](1-\p_x^2)^{-1}\right]^{-1}= \\
&=& 
\left( -\frac{2}{m''(0)} (1-\p_x^2)^{-1}+O_{B(L^2)}(\ep^2) \right) (1-\p_x^2)\left( \cl-\frac{2 \mu}{m''(0)}+O_{H^1}(\ep^2)\right)^{-1}=\\
&=& -\frac{2}{m''(0)} \left( \cl-\frac{2 \mu}{m''(0)}\right)^{-1}+O_{B(L^2)}(\ep^2). 
\end{eqnarray*} 
That is, 
\beq
\label{a:80}
(\ep^{-2} \cl_\ep+\mu)^{-1} =   \left( -\frac{m''(0)}{2} \cl+\mu \right)^{-1}+O_{B(L^2)}(\ep^2). 
\eeq
We can now use this  formula to study the spectrum of $\cl_\ep$. Using  min-max formulas for the eigenvalues of self-adjoint operators, we claim  that 
\beq
\label{a:90}
\la_{\max}((\ep^{-2} \cl_\ep+\mu)^{-1})=\la_{\max}\left(( -\frac{m''(0)}{2} \cl+\mu )^{-1}+O_{B(L^2)}(\ep^2)\right)>\frac{1}{\mu}
\eeq
for all small enough $\ep$.  

Indeed,  denoting the negative eigenvalue of $\cl$ by $-\si_0^2: \cl \psi_0=-\si_0^2 \psi_0, \|\psi_0\|=1$, we have that 
\begin{eqnarray*}
\la_{\max}\left(( -\frac{m''(0)}{2} \cl+\mu )^{-1}\right) &=& \sup_{f: \|f\|=1} \dpr{(-\frac{m''(0)}{2} \cl+\mu )^{-1} f}{f} \geq \dpr{(-\frac{m''(0)}{2} \cl+\mu )^{-1} \psi_0}{\psi_0}=\\
&=& \frac{1}{-\frac{m''(0)}{2} (-\si_0^2)+\mu}>\frac{1}{\mu}. 
\end{eqnarray*} 
It follows that for all small enough $\ep$, $\la_{\max}((\ep^{-2} \cl_\ep+\mu)^{-1})>\frac{1}{\mu}$, or equivalently, $\ep^{-2} \cl_\ep$ 
has the smallest eigenvalue in the form $\la_0(\ep^{-2} \cl_\ep):=\frac{1}{\la_{\max}\left(( -\frac{m''(0)}{2} \cl+\mu )^{-1}\right)}-\mu+O(\ep^2)<0$.   

Take $f:f\perp \psi_0, \|f\|=1$. Since we have $\cl|_{\{\psi_0\}^\perp}\geq 0$ and $\cl[\sigma']=0$, 
$$
\frac{1}{\mu}=\dpr{(-\frac{m''(0)}{2} \cl+\mu )^{-1} \frac{\sigma'}{\|\sigma'\|}}{\frac{\sigma'}{\|\sigma'\|}} \leq \sup_{f\perp \psi_0, \|f\|=1} \dpr{(-\frac{m''(0)}{2} \cl+\mu )^{-1} f}{f} \leq \frac{1}{\mu}
$$
It follows that $\la_1((-\frac{m''(0)}{2} \cl+\mu )^{-1})=\frac{1}{\mu}$, whence  the second smallest  eigenvalue for $(\ep^{-2} \cl_\ep+\mu)^{-1}$ is of the form $\frac{1}{\mu}+O(\ep^2)$. Equivalently, the second smallest eigenvalue for 
$\ep^{-2} \cl_\ep$ is $\la_1(\ep^{-2} \cl_\ep)=O(\ep^2)$. 

Further, according to the spectral information for $\cl$, its  second eigenvalue is also simple, in particular, $\cl|_{span\{\psi_0, \sigma'\}^{\perp}}\geq \de Id>0$. Therefore, 
$$
 \sup_{f\perp \psi_0, f\perp \sigma', \|f\|=1} \dpr{(-\frac{m''(0)}{2} \cl+\mu )^{-1} f}{f} \leq \frac{1}{-\de \frac{m''(0)}{2} +\mu}. 
$$
Equivalently, 
$$
\la_2(\ep^{-2} \cl_\ep)\geq   -\de \frac{m''(0)}{2} +O(\ep^2) >0.
$$
All in all, we have shown 
\beq 
\label{a:100}
\la_0(\ep^{-2} \cl_\ep)<0, \ \ \la_1(\ep^{-2} \cl_\ep)=O(\ep^2), \ \ \la_2(\ep^{-2} \cl_\ep)\geq   -\de \frac{m''(0)}{2} +O(\ep^2). 
\eeq
A direct differentiation in $x$  in the profile equation \eqref{TWE2} shows that $[\nu-L_\ep-n'(\ep^2 W)]W'=0$ or equivalently, $0\in \sigma(\cl_\ep)$. 
This, combined with \eqref{a:100} shows that $\la_1(\ep^{-2} \cl_\ep)=0$. This finishes the proof of Lemma \ref{le:a10}. 
\end{proof}
It remains to finally verify \eqref{a:50}. Now that we know that $Ker[\cl_\ep]=span\{W_\ep'\}$, we conclude that $\cl_\ep$ is invertible on the even subspace 
$L^2_{even}$.  In fact, we may use the formula \eqref{a:80} with $\mu=0$. In addition, from \eqref{107}, we have 
\begin{eqnarray*}
\dpr{\cl_\ep^{-1}[W_\ep]}{W_\ep} &=& -\frac{2}{m''(0)} \ep^{-2} \dpr{ (\cl^{-1}+O_{B(L^2)}(\ep^2)[\sigma+O_{H^1}(\ep^2)}{\sigma+O_{H^1}(\ep^2)}=\\
&=& -\frac{2}{m''(0)} \ep^{-2}[\dpr{\cl^{-1} \sigma}{\sigma}+O(\ep^2)]. 
\end{eqnarray*}
The quantity $\dpr{\cl^{-1} \sigma}{\sigma}$ is well-known in the theory of stability for the corresponding KdV/NLS models. Its negativity is exactly in the same way equivalent to the (well-known) stability of the corresponding traveling/standing waves. It actually may be computed explicitly as follows. 

Consider \eqref{TWE5} and a function $W_\la:=\la^2 \sigma(\la \cdot), \la>0$. This solves
$$
-W_\la''+\la^2 W_\la - \ga W_\la^2=0.
$$
Taking a derivative in $\la$ and evaluating at $\la=1$ yields 
$$
\cl[\frac{d}{d\la} W_\la|_{\la=1}]=-2\sigma
$$
Thus, $\cl^{-1} \sigma=-\frac{1}{2} \frac{d}{d\la} W_\la|_{\la=1}= -\frac{1}{2}(2\sigma+x \sigma')$. It follows that 
$$
\dpr{\cl^{-1} \sigma}{\sigma}=-\frac{1}{2} \dpr{2\sigma+x \sigma'}{\sigma}=-\frac{3}{4} \|\sigma\|^2<0.
$$
Thus, the Vakhitov-Kolokolov condition \eqref{a:50} is verified and the proof of Theorem \ref{theo:10} is complete.

 \subsection{Stability of the periodic waves}
 
 The stability calculation for the periodic waves proceed in an identical fashion. The eignevalue problem is in the form \eqref{a:20}, where now the operators are acting on the corresponding periodic spaces  $H^s(\T_P)$.  In fact, noting  that for $\la\neq 0$, the right hand side $z$ is an exact derivative, 
  allows us to restrict the consideration of  \eqref{a:20} to the space $L^2_0(\T_P)=\{f\in L^2(\T_P): \int_{-P}^P f(x) dx=0\}$. The advantage of this is that now $\cj=\p_x$ is boundedly invertible, hence allowing for the results of \cite{KP} to kick in. In particular, spectral stability and non-degeneracy do imply orbital stability. 
  
   The instability index theory outlined in Section \ref{sec:3.2} applies. According to \eqref{a:40} and the analysis in Section \ref{sec:3.3} - \eqref{a:50} implies the  spectral stability. Moreover, Lemma \ref{le:a10} applies as well to the periodic waves. That is, the Morse index of $\cl_\epsilon$ is one and the wave is non-degenerate, in the sense that $Ker[\cl_\epsilon]=span[W'_{P,\epsilon}]$. The verification of \eqref{a:50} is reduced, in the same way, to the verification of the inequality $\dpr{\cl^{-1}_P \phi_P}{\phi_P}<0$. This quantity can be computed fairly precisely, in terms of elliptic functions,  but we will not do so here.  
   
   Instead, we remark that Theorem \ref{cnoidal} sets up the spectral/orbital stability of the waves $\Phi_P$ of the periodic KdV model back to the same quantity. That is, the spectral stability of $\Phi_P$ is equivalent to $\dpr{\cl^{-1}_P \phi_P}{\phi_P}<0$.  Since it is well-known that $\Phi_P$ are stable with respect to co-periodic perturbations\footnote{in fact, much more is known, namely the cnoidal waves are stable with respect to harmonic perturbations - that is perturbations with periods $2m P, m=1,2, \ldots$,  \cite{BD}, \cite{DK}}, see for example \cite{BD}, \cite{DK}, it follows that $\dpr{\cl^{-1}_P \phi_P}{\phi_P}<0$. By the invertibility of $\p_x$ and the non-degeneracy of $\cl_{P,\epsilon}$, we also conclude orbital stability for $W_{P,\epsilon}$.

\appendix 

\section{Assorted Proofs}\label{assorted proofs}
\begin{proof} {\it (Lemma \ref{hinge})}
 Take $K>k_*/\ep$. Then we clearly have
$$
\left \vert  {1 \over {1 \over 2} m''(0)(1+ K^2)}\right \vert \le {2 \ep^2 \over |m''(0)|k_*} \le C\ep^2.
$$
Since $K>k_* /\ep$, we know that $m(\ep K) \le m_1$ by \eqref{upperbound}. Thus we have
$$
m(0) - {1 \over 2} m''(0) \ep^2 - m(\ep K) > m(0) - m_1 > 0.
$$
Here we have used that face that $m''(0) < 0$, which is implied by \eqref{m2 bound}. This in turn implies:
$$
\left \vert  {\ep^2 \over m(0)-{1 \over 2} m''(0) \ep^2 - m(\ep K)} \right \vert \le {\ep^2 \over m(0) - m_1} \le C \ep^2.
$$
The triangle inequality gives:
\be\label{big K}
\sup_{|K| \ge k_*/\ep} \left \vert {\ep^2 \over m(0)-{1 \over 2} m''(0) \ep^2 - m(\ep K)} + {1 \over {1 \over 2}m''(0)(1+ K^2)}\right \vert  \le C\ep^2.
\ee

Now suppose that $|K| \le k_*/\ep$. We have
\begin{multline}\label{thing}
{\ep^2 \over m(0)- {1 \over 2}m''(0) \ep^2 - m(\ep K)} + {1 \over  {1 \over 2}m''(0)(1+ K^2)} \\= { m(0) +  {1 \over 2}m''(0) \ep^2 K^2- m(\ep K) \over
[m(0)- {1 \over 2}m''(0) \ep^2 - m(\ep K)][ {1 \over 2}m''(0)(1+ K^2)]}
\end{multline}
The fact that $m$ is even and $C^{3,1}$ implies, by way of Taylor's theorem, that there exists $C>0$ such that
$$
 |m(0) +  {1 \over 2}m''(0) k^2- m(k)|  \le C{k^4}
$$
when $|k| \le k_*$. This implies that
\be\label{top}
 |m(0) +  {1 \over 2}m''(0) \ep^2 K^2- m(\ep K)| \le C \ep^4 K^4
\ee
when $|K| \le k_*/\ep$.

The fundamental theorem of calculus implies that 
$$
m(0)- {1 \over 2}m''(0) \ep^2 - m(k) = -{1 \over 2} m''(0) \ep^2 - \int_0^k \int_0^s m''(\sigma) d\sigma ds.
$$
Here we used the fact that $m(k)$ is even.
Then we use \eqref{m2 bound} to see that
$$
m(0)- {1 \over 2}m''(0) \ep^2 - m(k) \ge -{1 \over 2} m''(0) \ep^2 - {1 \over 2} m_2 k^2
$$
so long as $|k| \le k_*$. Thus, for $|K| \le k_*/\ep$ we have
$$
m(0)- {1 \over 2}m''(0) \ep^2 - m(\ep K) \ge \ep^2\left(-{1 \over 2} m''(0)  - {1 \over 2} m_2 K^2\right)
$$
Since $m''(0)$ and $m_2$ are both negative this implies:
\be\label{bot1}
|m(0)- {1 \over 2}m''(0) \ep^2 - m(\ep K)| \ge C\ep^2\left(1+ K^2\right)
\ee
when $|K| \le k_*/\ep$.

Thus we can control the left hand side of  \eqref{thing} using \eqref{top} and \eqref{bot1} as:
$$
\left \vert {\ep^2 \over m(0)- {1 \over 2}m''(0) \ep^2 - m(\ep K)} + {1 \over  {1 \over 2}m''(0)(1+ K^2)} \right \vert  \le
{C \ep^2 K^4 \over (1+K^2)^2}
$$
when $K\le k_*/\ep$. 
Since $K^4/(1+K^2)^2 \le 1$ we have 
$$
\sup_{|K| \le k_*/\ep} \left \vert {\ep^2 \over m(0)- {1 \over 2}m''(0) \ep^2 - m(\ep K)} + {1 \over  {1 \over 2}m''(0)(1+ K^2)} \right \vert \le C\ep^2
$$

\end{proof}

\end{document}